\documentclass[11pt]{article}
\usepackage{amssymb}
\usepackage{latexsym}
\usepackage{amsthm}
\usepackage{amscd}
\usepackage{amsmath}
\usepackage{diagrams}
\usepackage{indentfirst}

\theoremstyle{definition}
\newtheorem{thm}{Theorem}[section]

\newtheorem{cor}[thm]{Corollary}
\newtheorem{defn-lem}[thm]{Definition-Lemma}

\newtheorem{prop}[thm]{Proposition}
\newtheorem{question}[thm]{Question}
\newtheorem{rem}[thm]{Remark}

\newtheorem{defn}[thm]{Definition}
\numberwithin{equation}{section}

\allowdisplaybreaks[4]

\def \N{{\mathbb N}}
\def \C{{\mathbb C}}

\def \P{\mathbb P}

\def\map#1.#2.{#1 \longrightarrow #2}
\def\rmap#1.#2.{#1 \dasharrow #2}

\DeclareMathOperator{\gen}{gen}
\DeclareMathOperator{\Spec}{Spec}

\DeclareMathOperator{\im}{Im}

\DeclareMathOperator{\Frac}{Frac}

\def\fb#1.{\underset #1 \to \times}
\def\pr#1.{\Bbb P^{#1}}
\def\ring#1.{\mathcal O_{#1}}
\def\mlist#1.#2.{{#1}_1,{#1}_2,\dots,{#1}_{#2}}

\def\uloopr#1{\ar@'{@+{[0,0]+(-4,5)} @+{[0,0]+(0,10)}
@+{[0,0]+(4,5)}}
  ^{#1}}
\def\dloopr#1{\ar@'{@+{[0,0]+(-4,-5)} @+{[0,0]+(0,-10)}
@+{[0,0]+(4,-5)}}
  _{#1}}

\def\rloopd#1{\ar@'{@+{[0,0]+(5,4)} @+{[0,0]+(10,0)}
@+{[0,0]+(5,-4)}}
  ^{#1}}
\def\lloopd#1{\ar@'{@+{[0,0]+(-5,4)} @+{[0,0]+(-10,0)}
@+{[0,0]+(-5,-4)}}
  _{#1}}

\long\def\ignore#1{}
\long\def\ignore#1{#1}
\title{Semi-stable fibrations of generic $p$-rank 0}
\author{Junmyeong Jang}
\date{email : jang3@math.purdue.edu}

\begin{document}

\maketitle
\begin{center}
Mathematics Subject Classification : 11G25, 14J20
\end{center}

\medskip
     \section{Introduction}
     Let $k$ be an algebraically closed field and $\pi : X \to C$ be a semi-stable fibration of
     a connected proper smooth surface to a connected proper smooth curve over $k$.
     If the base field $k$ is a subfield of $\C$, the filed of complex numbers, the following semi-positivity theorem holds.
     \\ \\
     $\bf{Theorem.}$(Semi-Positivity Theorem, Xiao) If $\pi :X \to C$ is a fibration of a proper smooth surface to a proper smooth
     curve over $\C$, then all the quotient bundles of $\pi _{*}\omega _{X/C}$ are of non-negative degree.
     \cite{X},p.1\\ \\
     In general the semi-positivity theorem is not valid
     over a field of positive characteristic. In \cite{MB}, Moret-Bailly constructed a semi-stable fibration
     $\pi _{M} : X_{M} \to \P ^{1}$ of fiber genus 2 such that $R^{1} \pi _{M*} \mathcal{O}_{X_{M}} =
     \mathcal{O}(1) \oplus \mathcal{O}(-p)$ where $p$ is the characteristic of the base filed.
     In the previous work,\cite{J} we have proved that for a semi-stable fibration
     $\pi : X \to C$, if the generic fiber is ordinary, then the semi-positivity
     theorem holds. Precisely, when the generic fiber of $\pi$ is ordinary, all the Harder-Narasimhan slopes of
     $R^{1} \pi _{*} \mathcal{O}_{X}$ are non-positive. For Moret-Bailly's example, the $p$-rank of the generic fiber
     is 0. In particular, every special fiber of $\pi _{M}$ is a supersingular smooth curve of genus 2 or a union of
     two supersingular elliptic curves which intersect at a point transversally. In this paper we prove the following theorem
     which generalizes the failure of the semi-positivity theorem for Moret-Bailly's example.
     \\ \\
     $\bf{Theorem \ 1.}$
     Let $\pi : X \to C$ be a non-isotrivial semi-stable fibration of proper smooth surface to a proper smooth curve over
     a field of positive characteristic. If the generic $p$-rank of $\pi$ is 0, then
     $F _{C} ^{n*} R^{1} \pi _{*} \mathcal{O}_{X}$ has a positive Harder-Narasimhan slope
     for a sufficiently large $n \in \N$. In particular, if the genus of $C$ is 0 or 1, $R^{1} \pi _{*} \mathcal{O}_{X}$
     has a positive Harder-Narasimhan slope.\\ \\
     As an application of the theorem, we obtain a result on a distribution of $p$-ranks of reductions of a certain non-closed point
     in the moduli space of curves over a number field.\\ \\
     $\bf{Corollary \ 2.8.}$
     Suppose that $\pi : X \to C$ is a non-isotrivial
     semi-stable fibration of base genus 0 or 1 defined over a number field $F$,
     and that $U \subseteq C$ is the smooth locus of $\pi$.
     $\pi$ defines a non-constant morphism $f:U \to \mathcal{M} _{g,F}$. Let $P$ be the image of
     the generic point of $U$ under $f$. Then the reduction $P _{\upsilon}$ is not contained in
     the $p$-rank 0 strata for almost all $\upsilon$.\\ \\
     This result can be considered as a variation of Serre's ordinary reduction conjecture. It is a weak statement since
     it is only about some non-closed points in the moduli space and 0 $p$-rank. But it is a somewhat interesting
     phenomenon that the semi-positivity theorem, which is concerned with a coherent module on a fiber of characteristic 0,
     encodes an information of $p$-ranks of the reductions.

     \section{Proof of Theorem 1.}
     We follow the terminology of $\cite{J}$. Let us recall the definition of a semi-stable curve.
     Let $k$ be an algebraically closed field and $C$ be a projective curve over $k$.
     \begin{defn}
     $C$ is (semi-)stable if
     \begin{itemize}
     \item[1.] It is connected and reduced.
     \item[2.] All the singular points are normal crossing.
     \item[3.] An irreducible component, which is isomorphic to $\P^{1}$, meets other components
     in at least 3(resp. 2) points.
     \end{itemize}
     \end{defn}
     For an arbitrary base scheme, we define a (semi-)stable curve as follows.
     \begin{defn}
     A proper flat morphism of relative dimension 1 of schemes
     \\ $\pi : X \to S$ is a (semi-)stable curve
     if
     every geometric fiber of $\pi$ is
     a (semi-)stable curve in the sense of definition 2.1.
     \end{defn}
     In this paper, we assume $\pi : X \to C$ is a generically smooth semi-stable fibration of a proper smooth surface to a proper smooth curve
     over a field $k$ unless it is stated otherwise.
     \begin{defn}
     For a generically smooth semi-stable fibration $\pi : X \to C$ defined over a 
     field of positive characteristic, the generic $p$-rank of $\pi$ is the $p$-rank of a geometric generic fiber of $\pi$.
     \end{defn} 
     \subsection{Self duality of $B\omega ^{1}$}
     Let $k$ be a perfect field of positive characteristic and $\pi : X \to C$ be a generically smooth semi-stable curve.
     Let $\omega^{1} _{X/C}$ be the relative dualizing line bundle for $\pi$.
     There is the canonical inclusion
     $i : \Omega^{1} _{X/C} \hookrightarrow \omega^{1} _{X/C}$. At a relative smooth point for $\pi$, $i$ is an isomorphic.
     On the other hand, at a relative singular point, where \'{e}tale locally $\pi$ is given by
     $$ \Spec k[x,y,t]/(xy-t) \to \Spec k[t],$$ $\omega^{1} _{X/C}$
     is a free module of rank 1 generated by
     $dx/x = -dy/y$ and $\Omega^{1} _{X/C}$ is a submodule of $\omega^{1} _{X/C}$ generated by $dx$ and $dy$ via $i$.
     Composing with the inclusion $i$, we have the differential morphism $d : \mathcal{O} _{X} \to \omega^{1} _{X/C}$.
     When $F _{X/C}$ is the relative Frobenius morphism for $\pi$,  $\pi$
     $$ \begin{diagram}
     X & \rTo ^{F_{X/C}} & X^{p} &
     \rTo  & X \\
      & \rdTo _{\pi} & \dTo > {\pi ^{p}}& & \dTo >{\pi} \\
      & & C & \rTo ^{F_{C}} & C,\\
      \end{diagram}$$
      $d: F_{X/C*} \mathcal{O}_{X} \to F_{X/C*} \omega^{1} _{X/C}$ is $\mathcal{O}_{X^{p}}$-linear.
      The kernel of $d$ is the image of $F_{X/C}^{*} : \mathcal{O}_{X^{p}} \hookrightarrow F_{X/C*} \mathcal{O}_{X}$
      and the image of $d$ is denoted by $B^{1} \omega _{X/C}$ or $B^{1} \omega$.
      $B \omega ^{1}$ is flat over $\mathcal{O} _{C}$.
      Let $U \hookrightarrow X$ be the smooth locus for $\pi$. The usual Cartier isomorphism
      $$C :  \Omega^{1} _{U/C}/ B^{1}\Omega _{U/C} \rightarrow \Omega _{U^{p} / C}$$
      is extended to an isomorphism
      $$ C :  \omega ^{1} _{X/C}/ B^{1}\omega \to \omega^{1} _{X^{p}/C}. \textrm{\cite{I2},p.381}$$
      Using this Cartier morphism, we have an $\mathcal{O}_{X^{p}}$-linear paring
      $$F_{X/C *} \mathcal{O}_{X} \otimes  F_{X/C *} \omega ^{1} _{X/C} \to \omega ^{1} _{X^{p}/C}, \
      (\alpha, \omega) \mapsto C(\alpha \omega).$$
      This pairing induces a pairing
      $$ (F_{X/C *} \mathcal{O}_{X} / \mathcal{O}_{X^{p}}) \otimes  B^{1}\omega \to \omega ^{1} _{X^{p}/C}.$$
      On $U$, this pairing gives a perfect self duality.
      In particular, we have
     $$ B^{1} \Omega _{U/C} \simeq \underline{Hom}(B^{1} \Omega _{U/C}, \Omega ^{1} _{U^{p} /C}).$$
     \begin{prop}
     If $X$ is a smooth surface over a perfect field $k$ which admits a semi-stable
     fibration, $\pi :X \to C$, to a smooth curve $C$ over $k$, then
     $$B^{1} \omega _{X/C} \simeq \underline{Hom}(B^{1} \omega _{X/C}, \omega ^{1} _{X^{p}
     /C}).$$
     \end{prop}
     \begin{proof}
     It's enough to check that the paring is perfect at the relative singular points in
     $X$. Let $x \in X$ be a relative singular point. \'{E}tale
     locally, we may assume $X=\Spec A$,
     $C=\Spec B$ where
     $$A=k[x,y,t]/(xy-t) \simeq k[x,y],\ B=k[t]$$
     and $\pi$ is the canonical morphism
     $$k[t] \to k[x,y,t]/(xy-t).$$
     Let
     $A^{p}= A \otimes _{B} (B,F_{B})$. Then $A^{p} =
     k[X,Y,t]/(XY-t^{p})$ and the relative Frobenius morphism is a $k$-algebra morphism $F_{A/B} :
     A^{p} \to A$ given by
     $$X \mapsto x^{p}, \ Y \mapsto y^{p} \textrm{ and }t \mapsto xy.$$
     We may regard $A^{p}$ is a $k$-subalgebra of $A$
     generated by $x^{p},y^{p},xy$. As an $A^{p}$-module,
     $B^{1}\omega= A/A^{p}$ is generated by
     $$x, x^{2}, \cdots ,
     x^{p-1}, y, \cdots , y^{p-1}.$$
     $B^{1} \omega$ is a torsion free
     $A^{p}$-module and the $\Frac (A^{p})$-dimension of $B^{1} \otimes
     _{A^{p}} \Frac (A^{p})$ is $p-1$.
     Therefore there are only $p-1$ obvious relations
     $$ t^{p-1}x = X y^{p-1}, \ t^{p-2}x^{2} =X y^{p-2}, \cdots ,
     t x^{p-1} = X y$$
     among the generators and we have an $A^{p}$-module decomposition
     $$B^{1} \omega = \oplus _{i=1}^{p-1} <x^{i},y^{p-i}>.$$
     On the other hand, $\omega ^{1} _{A/B}$ is a rank 1 free $A$-module generated by
     $dx/x = -dy/y$ and $\omega ^{1} _{A^{p}/B}$ is a rank 1 free
     $A^{p}$-module generated by $dX / X = -dY/Y$. The Cartier
     morphism is the $A^{p}$-linear morphism satisfying
     $$\begin{array}{ccc}

     dx/x & \mapsto & dX/X,   \\
     xdx/x & \mapsto & 0, \\
     & \vdots \\
     x^{p-1}dx/x & \mapsto & 0,\\
     ydx/x & \mapsto & 0, \\
     & \vdots \\
     y^{p-1} dx/x & \mapsto & 0.
     \end{array}$$

     In the above decomposition of $B^{1}\omega$, it's easy to see that the dual of $<x^{i},
     y^{p-i}>$ is $<x^{p-i},y^{i}>$ and that the
     pairing $B^{1}\omega \otimes B^{1}\omega \to \omega ^{1}
     _{X^{p}/C}$
     gives a perfect duality of the dual components of
     both sides. This proves the claim.
     \end{proof}
     \begin{cor}
     Let $k$ be a perfect field of positive characteristic.
     Let $\pi : X \to C$ be a semi-stable fibration of a proper smooth surface
     to a proper smooth curve over $k$. Let $M$ and $T$ be the free part and the torsion part of
     $R^{1} \pi _{*} B^{1} \omega _{X/C}$
     respectively and $N= R^{1} \pi _{*} \mathcal{O}_{X}$.
     Then there exists a exact sequence of coherent modules on $C$
     $$ 0 \to M^{*} \to F_{C}^{*} N \to N \to M \oplus T \to 0.$$
           \end{cor}
     \begin{proof}
     The exact sequence of coherent $\mathcal{O} _{X^{p}}$-modules
     $$ 0 \to \mathcal{O} _{X^{p}} \to F_{X/C*} \mathcal{O} _{X} \to B^{1} \omega _{X/C} \to 0$$
     gives a long exact sequence for the $\pi ^{p} _{*}$ functor
     $$ 0 \to \mathcal{O}_{C} \cong \mathcal{O}_{C} \to \pi^{p} _{*} B^{1}\omega _{X/C}
     \to R^{1} \pi^{p} _{*} \mathcal{O}_{X^{p}} \to R^{1} \pi^{p} _{*} F_{X/C*} \mathcal{O}_{X} \to R^{1} \pi^{p} _{*}
     B^{1} \omega _{X/C} \to 0.$$
     Because the Frobenius morphism of $C$ is finite flat, $R^{1} \pi ^{p} _{*} \mathcal{O}_{X^{p}} =
     F_{C} ^{*} R^{1} \mathcal{O}_{X}$. And the relative Frobenius morphism $F : X \to X^{p}$ is
     finite affine, so $R^{1} \pi ^{p} _{*} F_{X/C*} \mathcal{O}_{X} = R^{1} \pi _{*} \mathcal{O}_{X}$.
     By proposition 2.4, $B^{1} \omega _{X/C} \simeq \underline{Hom}(B^{1} \omega _{X/C},
     \omega _{X^{p}/C})$. Since $\pi$ is relative 1-dimensional, by the relative duality theorem
     $$\pi ^{p} _{*} B^{1} \omega _{X/C} = \underline{Hom}( R^{1} \pi ^{p} _{*} B^{1} \omega _{X/C} , \mathcal{O}_{C}).$$
     Therefore the claim follows.
     \end{proof}
     \subsection{Proof of the theorem}
     \noindent
     $\bf{Theorem \ 1.}$
     Let $\pi : X \to C$ be a non-isotrivial semi-stable fibration of proper smooth surface to a proper smooth curve over
     a field of positive characteristic. If the generic $p$-rank of $\pi$ is 0, then
     $F _{C} ^{n*} R^{1} \pi _{*} \mathcal{O}_{X}$ has a positive Harder-Narasimhan slope
     for a sufficiently large $n \in \N$. In particular, if the genus of $C$ is 0 or 1, $R^{1} \pi _{*} \mathcal{O}_{X}$
     has a positive Harder-Narasimhan slope.\\
     \begin{proof}
     The $n$-iterative relative Frobenius morphism in the diagram
     $$\begin{diagram}
     X & \rTo ^{F_{X/C}^{n}} & X^{p^{n}} & \rTo  & X\\
     & \rdTo & \dTo >{\pi ^{p^{n}}} & & \dTo >{\pi}\\
     & & C & \rTo ^{F ^{n} _{C}} & C
     \end{diagram}$$
     is the composition of relative Frobenius morphisms
     $$X \overset{F_{X/C}}{\to} X^{p} \overset{F_{X^{p}/C}}{\to} \cdots \to X^{p^{n-1}}
     \overset{F_{X^{p^{n-1}}/C}}{\longrightarrow} X^{p^{n}}.$$
     $F^{n}_{X/C}$ gives an exact sequence of coherent $\mathcal{O}_{X^{p^{n}}}$-modules
     $$(*) \ 0 \to \mathcal{O}_{X^{p^{n}}} \overset{F^{n*}}{\to} F^{n}_{*}\mathcal{O}_{X} \to E_{n} \to 0.$$
     Here $E_{n}$ is flat over $\mathcal{O}_{C}$ and
     $E_{1}=B^{1}\omega _{X/C}$.
     If we denote $N= R^{1}\pi _{*} \mathcal{O}_{X}$, $R^{1} \pi ^{p^{n}} _{*} \mathcal{O}_{X ^{p^{n}}} =
     F^{n*} _{C} N.$
     Let $\lambda _{n} : F^{n*}_{C} N \to N$ be the morphism induced by $F_{X/C}^{n*}$
     in $(*)$.
     Because the relative Frobenius morphism commutes with a base change and the Frobenius morphism
     of $C$ is flat, $\lambda _{n}$ is the composition of Frobenius pullbacks of \\
     $\lambda _{1} : F^{*} _{C} N \to N$,
     $$\lambda _{n} = \lambda _{1} \circ \cdots \circ F^{(n-2)*} _{C} \lambda _{1} \circ
     F^{(n-1)*} _{C} \lambda _{1} :
     F^{n*}_{C} N \to F^{(n-1)*}_{C} N \to \cdots \to F^{*} _{C} N \to N.$$
     On the other hand, the restriction of $(*)$ to a special fiber $X_{s}$ is
     $$ 0 \to \mathcal{O}_{X_{s}^{p^{n}}} \to F^{n} _{X_{s}/k(s)*} \mathcal{O}_{X_{s}} \to E_{n,s} \to 0.$$
     Furthermore we have the long exact sequence
     $$ \cdots \to H^{1}(\mathcal{O}_{X_{s} ^{p^{n}}}) \to H^{1}(\mathcal{O}_{X _{s}})
     \to H^{1}(E_{n,s}) \to 0.$$
     Since we have assumed that the generic $p$-rank of $\pi$ is 0, by the
     Grothendieck specialization theorem, for all $s \in C$, the $p$-rank of $X_{s}$ is 0.
     It follows that there exists $n \in \N$, such that
     $H^{1}(X_{s}, \mathcal{O}_{X _{s}^{p^{n}}}) \to H^{1}(X_{s}, \mathcal{O}_{X_{s}})$ is the zero morphism
     for all $s \in C$ such that $X_{s}$ is smooth. Hence $\dim H^{0}(X_{s}, E_{n,s})= \dim H^{1}(X_{s}, E_{n,s}) = \gen X_{s}=g$
     for such $s$.
     But by the semi-continuity theorem, $\pi ^{p^{n}} _{*} E_{n}$ and $R^{1} \pi ^{p^{n}} _{*} E_{n}$ are vector bundles of rank $g$.
     Because $N$ is also a vector bundle of rank $g$ on $C$, considering the exact sequence
     $$ 0 \to \pi ^{p^{n}} _{*} E_{n} \to F^{n*}_{C} N
     \overset{\lambda _{n}}{\to} N \to R^{1} \pi ^{p^{n}} _{*} E_{n} \to 0,$$
     $\lambda _{n}=0$.
     Now assume all the Harder-Narasimhan slopes of $F^{i} _{C}N$ are non-positive for all $i$.
     Let $M$ be the free part of $R^{1} \pi ^{p} _{*} B^{1} \omega _{X/C}$.
     Then $\pi ^{p} _{*}
     B^{1} \omega _{X/C} = M^{*}= \underline{Hom}_{\mathcal{O}_{C}}(M,\mathcal{O}_{C})$.(Corollary 2.4)
     All the Harder-Narasimhan
     slopes of $F^{i*}_{C} M^{*}$ are non-positive since $F^{i*} _{C} M^{*}$ is a subbundle of
     $F^{i*} _{C} N$, so
     all the Harder-Narasimhan slopes of
     $F^{i*} _{C} M$ are non-negative.  Let us consider an exact sequence
     $$ 0 \to \im \lambda _{i-1}/ \im \lambda _{i} \to N / \im \lambda _{i} \to N/ \im \lambda _{i-1} \to 0.
     (i \geq 2)$$
     We can also think of the free part of the above exact sequence
     $$0 \to V' _{i} \to V_{i} \to V''_{i} \to 0.$$
     Here $V_{i}$ is the free part of $N / \im \lambda _{i}$ and $V'' _{i}$ is the free part of
     $N / \im \lambda _{i-1}$. $V' _{i}$ is the saturation of the free part of $\im \lambda _{i-1} /
     \im \lambda _{i}$ in $V _{i}$. We can see $V'' _{2} = M$ is of non-negative degree
     by the assumption. On the other hand,
     since $V' _{i}$ is a saturation of a quotient bundle of $F^{(i-1)*} M$, it is also of non-negative degree.
     Therefore, by induction $V _{i}$ is of non-negative degree. Since $\lambda _{n} =0$ for a
     sufficiently large $n$, the degree of $V_{n} = N$ is non-negative. 
     But since $\pi$ is non-isotrivial, $\deg N$ is strictly negative.\cite{SZ1},p.173
     It is contradiction, so $F^{i*}N$ has a positive Harder-Narasimhan slope for some $i$.
     If the genus of the base $C$ is 0 or 1, the Frobenius pull back $F_{C}^{*}$ preserves the
     semi-stability of vector bundles, so $R^{1} \pi _{*} \mathcal{O}_{X}$ has a positive Harder-Narasimhan
     slope.
     \end{proof}
     \begin{rem}
     The failure of the semi-positivity theorem for Moret-Bailly's example in the introduction is a special case of
     Theorem 2. Using Theorem 2, we may construct a lot of counterexamples of the semi-positivity theorem.
     Assume $k$ is algebraically closed of positive characteristic.
     In $\mathcal{M} _{g,k}$, the moduli space of smooth proper curves of genus $g$ over $k$, the $p$-rank
     0 strata is a closed subscheme which is purely $2g-3$ dimensional.\cite{FV},p.120
     Let $P$ be a 1 dimensional point in the $p$-rank 0
     strata. By the semi-stable reduction theorem,\cite{DE},p.3 
     there is a semi-stable fibration $\pi : X \to C$ such that the morphism
     $C \to \mathcal{M}_{g,k}$ induced by $\pi$ sends the generic point of $C$ to $P$.
     Since the generic $p$-rank of $\pi$ is 0,
     for a suitable Frobenius base change $\pi ^{p^{n}} : X ^{p^{n}} \to C$, $R^{1} \pi ^{p^{n}} _{*} \mathcal{O}
     _{X ^{p^{n}}}$ has a positive slope by Theorem 1. $X^{p^{n}}$ may contain isolated singularities.
     But the composition of $\pi ^{p^{n}}$ and the desingularization $X^{(n)} \to X^{p^{n}}$,
     $\pi ^{(n)} : X^{(n)} \to C$ is a semi-stable fibration and
     $R^{1} \pi ^{(n)} _{*} \mathcal{O} _{X^{(n)}} = R^{1} \pi ^{p^{n}} _{*} \mathcal{O}_{X^{p^{n}}}$.
     Hence $\pi ^{(n)}$ is a counterexample to the semi-positivity theorem.
     \end{rem}
     \begin{cor}
     Let $F$ be a number field and
     suppose a semi-stable fibration $\pi : X \to C$ is defined over $F$.
     There is an integral model of $\pi$, $\pi _{A} : X_{A} \to C_{A}$ defined over
     $\Spec A$, an affine open set of $\Spec \mathcal{O}_{F}$. Let $\pi _{\upsilon} : X _{\upsilon} \to
     C_{\upsilon}$ be the reduction of $\pi _{A}$ at a place $\upsilon \in \Spec A$.
     If the genus of $C$ is 0 or 1, then the generic $p$-rank of $\pi _{\upsilon}$ is not 0
     for all but finitely many places $\upsilon$.
     \end{cor}
     \begin{proof}
     Since the harder-Narasimhan filtration of $R^{1} \pi _{A*} \mathcal{O}_{X_{A}}$ on the generic fiber of $\C _{A} \to \Spec A$ extends to
     a non-empty open set of $\Spec A$,  $R^{1} \pi _{\upsilon *}
     \mathcal{O}_{X _{\upsilon}}$ has no positive Harder-Narasimhan slope for almost all $\upsilon \in 
     \Spec A$. Therefore the generic $p$-rank of
     $\pi _{\upsilon}$ is not 0 by Theorem 1.
     \end{proof}

     Let $\mathcal{M}
     _{g, \mathcal{O} _{F}}$ be the moduli space of proper smooth curves over $\Spec \mathcal{O}_{F}$.
     $\mathcal{M} _{g,F}$, the moduli space over $F$, is the generic fiber of
     $\mathcal{M} _{g, \mathcal{O}_{F}} \to \Spec \mathcal{O}_{F}$.
     When $P$ is a geometrically irreducible point of $\mathcal{M} _{g,F}$,
     the closure of $P$ in $\mathcal{M} _{g, \mathcal{O}_{F}}$
     has a geometrically irreducible reduction at almost all places $\upsilon$. Let us denote the generic point
     of the reduction at $\upsilon$ by $P _{\upsilon}$. $P _{\upsilon}$ is contained in a
     $p$-rank strata in $\mathcal{O} _{g, \bar{k_{\upsilon}}}$, the moduli space over the residue field at $\upsilon$.
     We may ask the distribution of the $p$-ranks of $P _{\upsilon}$. Serre's ordinary reduction conjecture is
     a problem for closed points in the moduli space.
     In the language of the moduli space, Corollary 2.7 can be stated as follow.
     \begin{cor}
     Suppose that $\pi : X \to C$ is a non-isotrivial
     semi-stable fibration of base genus 0 or 1 defined over a number field $F$,
     and that $U \subseteq C$ is the smooth locus of $\pi$.
     $\pi$ defines a non-constant morphism $f:U \to \mathcal{M} _{g,F}$. Let $P$ be the image of
     the generic point of $U$ under $f$. Then the reduction $P _{\upsilon}$ is not contained in
     the $p$-rank 0 strata for almost all $\upsilon$.
     \end{cor}
     \begin{rem}
     If the fiber genus of $\pi : X \to C$ is 2, Corollary 2.7 holds for arbitrary base $C$.
     Ekedahl showed that if $\pi : X \to C$ is a generically super-singular semi-stable fibration of fiber genus 2,
     then there exist a finite \'{e}tale cover $f: D \to C$ and a morphism $g : D \to \P ^{1}$ such that
     $\pi _{f} : X \times _{C} D \to D$ is isomorphic to $\pi _{M , g} : X_{M} \times _{\P ^{1}} D \to D$ where
     $\pi _{M} : X_{M} \to \P ^{1}$ is Moret-Bailly's fibration.\cite{EK},p.173
     Since a pullback by a finite separable morphism
     of curves preserves the semi-stability of vector bundles, $R^{1} \pi _{*} \mathcal {O}_{X}$ has
     a positive Harder-Narasimhan slope. Considering
     the construction of a semi-stable fibration from a 1-dimensional point in the moduli space(Remark 2.6)
     and the Grothendieck specialization theorem,
     Corollary 2.8 holds for an arbitrary
     non-closed point in $\mathcal{M} _{2,F}$.
     \end{rem}

     It is natural to expect that Corollary 2.8 holds for an arbitrary non closed point in $\mathcal{M} _{g,F}$ for
     any $g$, or equivalently that Corollary 2.7 holds for arbitrary base curve $C$.
     If the result of Proposition 2.12 in \cite{J} is valid over a filed of characteristic 0,
     i.e. the slope 0 part of $R^{1} \pi _{*} \mathcal{O}_{X}$ is potentially trivial,
     this expectation is valid.
     Indeed, in the situation of Corollary 2.7 without the assumption of the base genus, if 
     the slope 0 part, $(R^{1}
     \pi _{*} \mathcal{O} _{X}) _{0}$ is potentially trivial,
     $(R^{1} \pi _{\upsilon} \mathcal{O} _{X _{\upsilon}})_{0}$ is strongly semi-stable
     for almost all $\upsilon \in \Spec A$. On the other hand,
     by \cite{SB},p.660, the negative slope part, $F_{\upsilon} ^{n*} (R^{1} \pi _{*} \mathcal{O}_{X _{\upsilon}})_{-}$ has only negative
     Harder-Narasimhan slopes for any $n$ if the residue characteristic of $\upsilon$ is sufficiently large.
     Therefore the claim follows.

     \begin{question}
     For an arbitrary geometrically irreducible non-closed point $P$ in $\mathcal{M}_{g,F}$,
     is the $p$-rank of the reduction $P _{\upsilon}$ nonzero
     for almost all $\upsilon \in \Spec \mathcal{O}_{F}$?
     \end{question}

     \end{document}